\documentclass{amsart} 

\usepackage[french]{babel}
\usepackage[latin1]{inputenc}

\usepackage{amsmath}
\usepackage{amsthm}
\usepackage{amssymb}
\usepackage{tikz}
\usepackage[all]{xy} 

\usepackage{imakeidx} 
\usepackage{appendix}
\usepackage{tikz-cd}
\usepackage{verbatim}
\usepackage{enumitem}
\usepackage{multicol}

\usepackage[backref=page]{hyperref}
\usepackage{cleveref}

 \hypersetup{colorlinks=true,linkcolor=blue,anchorcolor=blue,citecolor=blue}

\usepackage{adjustbox}

\usepackage{mathtools}

\usepackage{amscd}

\usepackage{array}
\usepackage{xy}
\xyoption{all}
\usepackage{mathtools}

\usepackage{imakeidx} 
\usepackage{verbatim}
\usepackage{enumitem}
\usepackage{multicol}

\usepackage[backref=page]{hyperref}
\usepackage{cleveref}

 \hypersetup{colorlinks=true,linkcolor=blue,anchorcolor=blue,citecolor=blue}

\usepackage{adjustbox}

\usepackage{mathtools}

\newtheorem{thm}{Th\'eor\`eme}[section]
\newtheorem{prop}[thm]{Proposition}
\newtheorem{lemma}[thm]{Lemme}
\newtheorem{cor}[thm]{Corollaire}

\newtheorem{question}[thm]{Question}

\theoremstyle{definition}
\newtheorem{defin}[thm]{D\'efinition}

\theoremstyle{remark}

\newtheorem{rmk}[thm]{Remarque}
\newtheorem{remark}[thm]{Remarque}
\numberwithin{equation}{section} 

\newcommand{\Coker}{{\rm Coker}}
 \newcommand{\Ker}{{\rm Ker}}
\newcommand{\Gal}{{\rm Gal}}
\newcommand{\Br}{{\rm Br}}
\newcommand{\Q}{\mathbb Q}
\newcommand{\R}{\mathbb R}
\newcommand{\F}{\mathbb F}
\newcommand{\C}{\mathbb C}
\newcommand{\Z}{\mathbb Z}

\renewcommand{\P}{\mathbb P}

\renewcommand{\phi}{\varphi}

\def\im{{\rm im}}
\def\Pic{{\rm Pic}}
\def\Alb{{\rm Alb}}
\def\Br{{\rm Br}}
\def\k{{\overline k}}

\def\Gal{{\rm Gal}}
\def\X{{\overline X}}
\def\P{{\bf P}}
\def\X{{\overline X}}
\def\B{{\overline B}}
\def\k{{\overline k}}
\def\NS{{\rm NS}}
\def\Sym{{\rm Sym}}

 \author{J.-L. Colliot-Th\'el\`ene}

\title[Application d'Albanese et z\'ero-cycles]{Notes sur l'application d'Albanese pour les z\'ero-cycles}

\address{Universit\'e Paris-Saclay, CNRS, Laboratoire de math\'ematiques d'Orsay, 91405, Orsay, France.}
\email{jean-louis.colliot-thelene@universite-paris-saclay.fr}

\date{
8 juin 2025}

\begin{document}
	\maketitle
	\hypersetup{backref=true}

\begin{abstract} 
Pour $X$ une vari\'et\'e projective et lisse sur un corps $k$,
on a un homomorphisme du groupe de Chow $A_{0}(X)$ 
des z\'ero-cycles de degr\'e z\'ero
vers le groupe des points $k$-rationnels $\Alb_{X}(k)$
de la vari\'et\'e d'Albanese de $X$.
On discute la question de la surjectivit\'e de cette application.
Pour $k$ corps $p$-adique ou r\'eel, on donne des exemples
de non surjectivit\'e. Pour $k=\mathbb C$ le corps des complexes, on consid\`ere 
l'application $A_{0}(X_{F}) \to \Alb_{X}(F)$ pour $F$ extension de $\mathbb C$
et en particulier $F$ \'egal au corps des fonctions de $Alb_{X}$. 
On fait le lien avec des travaux  r\'ecents de C. Voisin  sur  la notion
de z\'ero-cycle universel  et sur les cycles de codimension deux 
sur les solides rationnellement connexes.
 \end{abstract}

 \section{Introduction}
 
 Soit $k$ un corps.
 Dans cet article, sauf mention expresse du contraire, on suppose $k$ de caract\'eristique z\'ero.
On note $\k$ une cl\^{o}ture alg\'ebrique de $k$ et $G=G_{k}=\Gal(\k/k)$ le groupe de Galois.

Soit $X$ une $k$-vari\'et\'e propre. On note $Z^0(X)$ le groupe des z\'ero-cycles sur $X$,
qui est le groupe ab\'elien libre sur les points ferm\'es de $X$. Le 
degr\'e des points ferm\'es sur $k$  induit une application degr\'e $Z^0(X) \to \Z$.
On note $Z^0_{0}(X)$ son noyau.
On note $CH_{0}(X)$ le groupe de Chow des z\'ero-cycles
de degr\'e z\'ero  modulo l'\'equivalence rationnelle et $A_{0}(X) \subset CH_{0}(X)$ le sous-groupe  des classes
de z\'ero-cycles de degr\'e z\'ero.

Soit $X$ une $k$-vari\'et\'e projective et lisse g\'eom\'etriquement connexe. Soit $\X=X\times_k\k$.
Soit $\Alb_{X}= \Alb_{X/k}$ la vari\'et\'e d'Alba\-nese de $X$. C'est une vari\'et\'e ab\'elienne sur $k$.
 Il y a un torseur $E:=\Alb^1_{X}$
sous $ \Alb_{X}$ et un  $k$-morphisme naturel $\phi: X \to E$.
La vari\'et\'e d'Albanese de $E$ s'identifie \`a $A$.
On consultera \cite{S59},  \cite[Thm. 3.3]{Gr62}, \cite{Kl05}, \cite{W08}.

Soit $K$ un corps alg\'ebriquement clos.
Soit  $C/K$ une courbe connexe, projective, lisse. Soit $ J_{C} : = \Pic^0_{C/K}$ la jacobienne de $C$.
C'est une vari\'et\'e ab\'elienne. On a un morphisme $C \to J_{C}$ associ\'e au choix d'un point $m$ de $C(K)$.
Il induit un isomorphisme $A_{0}(C) \simeq J_{C}(K)$, qui ne d\'epend pas du choix de $m$.
 Soit $A/K$ une vari\'et\'e ab\'elienne.
Soit    $f : C \to A$
un $K$-morphisme envoyant le point $m \in C(K)$ sur le point $0 \in A(K)$.
Cette application se factorise :
$$C \to J_{C} \to A,$$ avec $J_{C} \to A$
un homomorphisme de vari\'et\'es ab\'eliennes. 
On obtient ainsi un homomorphisme $A_{0}(C) \to A(K)$
ind\'ependant du choix de $m$.
 Soit $E/K$ un torseur sous $A$.
\`{A} tout z\'ero-cycle  $\sum_{i}n_{i}P_{i} $  de degr\'e z\'ero sur $E$, avec $n_{i} \in \Z$
et $P_{i}\in E(K)$,
on associe le point  $\sum'_{i} n_{i}P_{i}  \in A(K)$, o\`u la somme
est prise via la structure de torseur de $E$ sous $A$.  
Pour $X/k$ une $k$-vari\'et\'e propre lisse g\'eom\'etriquement int\`egre,
prenant $K=\k$, ceci induit
 un homomorphisme  Galois-\'equivariant
$$  Z^0_{0}(\X) \to Z^0_{0}({\overline E}) \to A(\k) =\Alb_{X}(\k).$$
La d\'efinition de l'\'equivalence rationnelle et le cas des courbes
montre que cet homomorphisme induit un   homomorphisme Galois-\'equivariant
$$A_{0}(\X) \to  \Alb_{X}(\k).$$
On en d\'eduit un homomorphisme
$$A_{0}(X) \to A_{0}(\X)^G \to\Alb_{X}(\k)^G =\Alb_{X}(k)$$

Cet homomorphisme est fonctoriel en le corps $k$
et fonctoriel covariant en la $k$-vari\'et\'e $X$.

\begin{prop}\label{roitman}
(1) L'homomorphisme 
$A_{0}(X) \to A_{0}(\X)^G$
a noyau et conoyau
de torsion.

(2) (Roitman) Le noyau de  l'homomorphisme  surjectif $A_{0}(\X) \to\Alb_{X}(\k)$
est uniquement divisible.

(3) L'homomorphisme $A_{0}(\X) \to \Alb_{X}(\k)$ induit un
  homomorphisme surjectif
$A_{0}(\X)^G \to \Alb(\X)^G=\Alb_{X}(k)$
\`a   noyau uniquement divisible.

(4) Le conoyau de l'homomorphisme compos\'e
$$A_{0}(X) \to  A_{0}(\X)^G \to  \Alb_{X}({\overline k})^G=\Alb_{X}(k)$$
est de torsion. 

(5) (B. Kahn) ll existe un entier $n(X)$ tel que pour tout corps $F$
contenant $k$, le conoyau de $A_{0}(X_{F}) \to \Alb_{X}(F)$
soit annul\'e par $n(X)$.

\end{prop}

\begin{proof}
L'\'enonc\'e (1) r\'esulte de l'\'enonc\'e analogue pour
la fl\`eche $$A_{0}(X) \to A_{0}(X_{K})^{\Gal(K/k)}$$ pour
$K/k$ fini galoisien. Ce dernier se voit en utilisant les propri\'et\'es de
l'application norme $A_{0}(X_{K}) \to A_{0}(X)$.
L'\'enonc\'e (2) est un th\'eor\`eme de Roitman \cite{R80}  selon lequel l'application
$A_{0}(\X) \to\Alb_{X}(\k)$  identifie 
 la torsion de  $A_{0}(\X)$ avec la torsion de $ \Alb_{X}(\k)$.
On a donc une suite exacte
$$0 \to V \to A_{0}(\X) \to \Alb_{X}(\k) \to 0$$
avec $V$ un $\Q$-vectoriel, donc satisfaisant $H^1(G,V)=0$.
La suite exacte de cohomologie galoisienne donne
l'\'enonc\'e (3). L'\'enonc\'e (4) suit de (1) et (3).

L'\'enonc\'e (5) est \'etabli par Bruno Kahn \cite[Prop. A.1]{K21}.
Indiquons  l'id\'ee de la d\'emonstration. 
On consid\`ere le corps $E=k(\Alb_{X})$.
Soit $\eta \in \Alb_{X}(E)$ le point g\'en\'erique.
D'apr\`es l'\'enonc\'e (4), il existe un entier $N=n(X)$
et un z\'ero-cycle $z$ de degr\'e z\'ero dans $Z_{0}(X_{E})$
tel que $N(\eta-0_{E})$ soit image de $z$.
Tout point de $\Alb_{X}(F)$ s'obtient par 
sp\'ecialisation \`a partir de $\eta \in \Alb_{X}(E)$.
Ceci donne le r\'esultat. \footnote{Pour une d\'emonstration plus d\'etaill\'ee
d'un r\'esultat semblable, voir la proposition \ref{zerocycleuniv} ci-dessous.}
 \end{proof}

Sous l'hypoth\`ese que  la $k$-vari\'et\'e $X$ poss\`ede un point $k$-rationnel,
ou du moins poss\`ede un z\'ero-cycle de degr\'e 1, on peut se poser les questions suivantes :

(a) L'homomorphisme $A_{0}(X) \to A_{0}(\X)^G$ est-il surjectif ?

(b) Pour tout corps $L$ contenant $k$, l'homomorphisme
$A_{0}(X_{L}) \to A_{0}(X_{\overline L})^{G_{L}}$ est-il surjectif ?

(c)  Pour tout corps $L$ contenant $k$, l'homomorphisme
$A_{0}(X_{L}) \to \Alb_{X}(L)$
est-il surjectif ?

\bigskip

Dans ce texte, nous rassemblons des r\'esultats divers sur ces probl\`emes.

On donne des contre-exemples \`a la surjectivit\'e parmi les vari\'et\'es
de Severi-Brauer au-dessus d'une courbe, en particulier
sur des corps ``arithm\'etiques'', comme les corps $p$-adiques ou
le corps des r\'eels.

On s'int\'eresse par ailleurs au
cas o\`u le corps de base $k$ est le corps $\C$ des complexes
et $L$ varie parmi les corps de fonctions de vari\'et\'es sur $\C$.
On consid\`ere tout particuli\`erement le cas des solides\footnote{Un ``solide''  est une vari\'et\'e int\`egre de dimension 3} $X/\C$
qui sont rationnellement connexes (th\'eor\`eme  \ref{equivgenerales})
et des hypersurfaces cubiques de dimension 3 (th\'eor\`eme \ref{equivcubiques}).

\section{Fibrations}

 \begin{lemma} \label{isoalb}
Soit  $f: X \to B$
 un $k$-morphisme dominant  de $k$-vari\'et\'es projectives, lisses,
 g\'eom\'etriquement int\`egres, \`a fibre g\'en\'erique
 g\'eom\'etriquement int\`egre. Soit $Z/\k(B)$  la fibre g\'en\'erique de $\X \to \B$. 
 Supposons que  le groupe $\Pic(Z)$ 
 est de type fini.  Alors  la fl\`eche $f^* : \Pic^0_{B/k} \to \Pic^0_{X/k}$ est un isomorphisme,
 et la fl\`eche $f_{*} : \Alb_{X} \to \Alb_{B}$ est un isomorphisme.
 \end{lemma}
{\begin{proof}  Les arguments sont bien connus. 
 Pour $f : X \to B$ morphisme propre de vari\'et\'es int\`egres \`a fibre g\'en\'erique
 g\'eom\'etriquement int\`egre,
  la fl\`eche ${\mathcal O}_{B}^{\times} \to f_{*} {\mathcal O}_{X}^{\times}$ est un isomorphisme,
  donc la fl\`eche  $\Pic(B) \to \Pic(X)$ est injective.  
  Pour \'etablir l'\'enonc\'e, on peut supposer que $k$ est alg\'ebriquement clos.
Il existe un ouvert non vide $U \subset B$   tel que le morphisme induit $X_{U} \to U$
soit lisse \`a fibres int\`egres.  On a le diagramme commutatif de suites exactes
	  \[
\begin{tikzcd}
0 \arrow[r] & K_{1} \arrow[r] \arrow[d] &   \Pic(B)  \arrow[r] \arrow[d] &   \Pic(U) \arrow[r] \arrow[d] & 0 \\
0 \arrow[r] &  K_{2} \arrow[r] & \Pic(X) \arrow[r]  & \Pic(X_{U}) \arrow[r] & 0. \end{tikzcd}
	\]	
Le groupe $K_{2}$ est de type fini. 
Le conoyau de $\Pic(U) \to \Pic( X_{U})$ est $\Pic(Z)$.
Le  groupe $\Pic(X)/\im(\Pic(B)$ est une extension de $\Pic(Z)$
 par un quotient de  $K_{2}$, donc de type fini. Comme $\Pic(Z)$ est de type fini,
 on conclut que $\Pic(X)/\im(\Pic(B)$ est de type fini.
 On a les suites exactes compatibles
  \[
\begin{tikzcd}
0 \arrow[r] & \Pic^0_{B/k}(k) \arrow[r] \arrow[d] &   \Pic(B)  \arrow[r] \arrow[d] &   \NS(B) \arrow[r] \arrow[d] & 0 \\
0 \arrow[r] & \Pic^0_{X/k}(k) \arrow[r] & \Pic(X) \arrow[r]  & \NS(X)  \arrow[r] & 0, \end{tikzcd}
	\]		
 o\`u les groupes de N\'eron-Severi sont de type fini.
On en d\'eduit que le conoyau de la fl\`eche injective
$$\Pic^0_{B/k}(k) \to \Pic^0_{X/k}(k)$$
est un groupe de type fini. Mais c'est un groupe divisible.
Il est donc nul.  Ainsi  $f^* : \Pic^0_{B/k} \to \Pic^0_{X/k}$ est un isomorphisme.

  \end{proof}

\begin{remark}  Si l'on a $H^1(Z,O_{Z})=0$, 
la condition $\Pic(Z)$ de type fini est satisfaite, car le groupe de N\'eron-Severi est
un groupe de type fini.  
 \end{remark}

\begin{prop}\label{ratconsurquelconque}  
 Supposons le corps $k$ alg\'ebriquement clos.
 Soit $f: X \to Y$ un $k$-morphisme de vari\'et\'es
 projectives et lisses connexes.
 Si la fibre g\'en\'erique de $f$ est une vari\'et\'e rationnellement connexe,
 alors l'application $f_{*} : A_{0}(X) \to A_{0}(Y)$
 est un isomorphisme.
\end{prop}

\begin{proof}  (Voisin) Par sections hyperplanes g\'en\'erales on peut trouver un plongement ferm\'e
$i : Z \subset X$ projectif et lisse tel que la projection 
$p : Z \to Y$ obtenue par composition $Z \subset X \to Y$ soit g\'en\'eriquement finie.
Soit $n\geq 1$ son degr\'e. Soit $U \subset Y$ ouvert non vide tel que $p^{-1}(U) \to U$ soit fini \'etale
et que les fibres de $X_{U} \to U$ soient lisses. Elles sont alors rationnellement connexes.
Tout z\'ero-cycle sur $X$ est rationnellement \'equivalent \`a un z\'ero-cycle \`a support
dans $X_{U}$. Comme tous les points ferm\'es des fibres $X_{y}$, $y\in U(k),$ sont
R-\'equivalents, on en conclut que
l'application  $i_{*} : CH_{0}(Z) \to CH_{0}(X)$ est surjective. 

On a l'application $p^* : CH_{0}(Y) \to CH_{0}(Z)$, 
et l'application compos\'ee  
$$\phi =i_{*}   \circ p^* : CH_{0}(Y) \to CH_{0}(Z) \to CH_{0}(X).$$
Pour $P \in Z_{y}, y \in U(k)$, le z\'ero-cycle $nP$ est rationnellement \'equivalent sur $X_{y}$
au z\'ero-cycle $p^{-1}(y)$. En utilisant le lemme de d\'eplacement, on conclut que
l'application $\phi$ a son conoyau annul\'e par $n$.

Par ailleurs le compos\'e de  $\phi$ avec la projection $CH_{0}(X) \to CH_{0}(Y)$
co\"{i}ncide avec le compos\'e de $CH_{0}(Y) \to CH_{0}(Z) \to CH_{0}(Y)$,
qui est la multiplication par $n$. Donc le noyau de $\phi$ est annul\'e par $n$.

On obtient ainsi  un homomorphisme   $\theta_{Z} : A_{0}(Y) \to A_{0}(X)$ (d\'ependant du choix de $Z \subset X$)
tel que   la composition
$$A_{0}(Y) \to A_{0}(X) \to A_{0}(Y)$$
soit la multiplication par $n$ et que le conoyau de $\theta_{Z} : A_{0}(Y) \to A_{0}(X)$ soit
annul\'e par $n$. On en  conclut que la fl\`eche surjective $f_{*} : A_{0}(X) \to A_{0}(Y)$
est un isomorphisme apr\`es tensorisation par $\Q$.

Plus pr\'ecis\'ement, soit $z$ dans le noyau de $f_{*}: A_{0}(X) \to A_{0}(Y)$.
On sait que l'on a $nz=\theta_{Z}(\rho)$ avec $\rho \in A_{0}(Y)$.
Alors $\rho$ est dans le noyau de $ n: A_{0}(Y) \to A_{0}(Y)$.
Donc $n\rho=0 \in A_{0}(Y)$. Donc $n^2z=\theta_{Z}(n\rho)=0 \in A_{0}(X)$.

D'apr\`es le lemme \ref{isoalb}, la fl\`eche $\Alb_{X}  \to \Alb_{B}$ est un isomorphisme.
Le th\'eor\`eme de Roitman assure que le noyau des applications (fonctorielles)  surjectives
$A_{0}(X) \to \Alb_{X}(k)$ et $A_{0}(Y) \to \Alb_{Y}(k)$ est un $\Q$-vectoriel.
On conclut que le noyau de l'application $A_{0}(X) \to A_{0}(Y)$ est un espace vectoriel sur $\Q$.
Comme ce noyau est de torsion, il est nul. \end{proof}

\begin{remark}  Des cas particuliers \'el\'ementaires 
de la proposition \ref{ratconsurquelconque} 
suffisent dans la plupart des exemples donn\'es plus loin. 
 Claire Voisin (6 octobre 2022)  m'a indiqu\'e  la d\'emonstration ci-dessus.
Des variantes de ce r\'esultat sont  en fait d\'ej\`a dans la litt\'erature,
avec des hypoth\`eses plus g\'en\'erales sur la fibre g\'en\'erique g\'eom\'etrique
(d\'ecomposition rationnelle de la diagonale).
Olivier Wittenberg m'indique ainsi que le th\'eor\`eme ci-dessus r\'esulte   de \cite[Lemme 2.3]{W12}. Bruno Kahn me signale le th\'eor\`eme \cite[Thm. 1.3]{Vial15} de Ch. Vial
et  aussi son r\'esultat \cite[Cor. 6.8 a)]{K18}.
\end{remark}

\begin{prop}\label{propprincip}
Soient  
 $B$ et $X$ des   $k$-vari\'et\'es projectives, lisses, g\'eom\'etriquement connexes, et
 $f: X \to B$  un $k$-morphisme \`a fibre g\'eom\'etrique g\'en\'erique  
 une vari\'et\'e rationnellement connexe.

 (a) Le morphisme $f$ induit le
  diagramme commutatif suivant :

	$$\xymatrix{
	A_{0}(X) \ar[r]  \ar[d]_{f_{*}}  &   A_{0}(\X)^G  \ar[r] \ar[d]^{\simeq}_{f_{*}} &  \Alb_{X}(k)  \ar[d]^{\simeq}_{f_{*}}  \\
	A_{0}(B) \ar[r]  &  A_{0}(\B)^G \ar[r] &  \Alb_{B}(k) 
	}$$
	 o\`u  les deux fl\`eches horizontales de droite sont  des fl\`eches surjectives  \`a noyau uniquement divisible.
 
  (b) Supposons que $A_{0}(\B) \to \Alb_{B}({\overline k})$
  est un isomorphisme, alors  
  $A_{0}(\X) \to  \Alb_{X}({\overline k})$
  est un isomorphisme.

(c)  Supposons de plus  que   $A_{0}(B) \to A_{0}(\B)^G$
est injectif.  
Si  $$A_{0}(X) \to  A_{0}(\X)^G \simeq \Alb_{X}(k)$$
  est surjectif, alors
 $f_{*} : A_{0}(X) \to A_{0}(B)$ est surjectif.

(d) Supposons de plus  que $A_{0}(B) \to A_{0}(\B)^G$ est un isomorphisme.
 Si  $$f_{*} : A_{0}(X) \to A_{0}(B)$$ est surjectif, alors 
 $A_{0}(X) \to  A_{0}(\X)^G$ est surjectif.

   \end{prop}

\begin{proof} On utilise la proposition  \ref{roitman}, le lemme \ref{isoalb} et la proposition  \ref{ratconsurquelconque}.
\end{proof}

\begin{prop}\label{severibcourbe}
Soient 
 $C$ une $k$-courbe projective, lisse, g\'eom\'etriquement connexe, et
 $f: X \to C$  un $C$-sch\'ema de Severi-Brauer.
 Soit $\alpha \in \Br(C)$ la classe associ\'ee. 
 S'il existe un z\'ero-cycle $z$ de degr\'e z\'ero sur $C$ tel que $ \alpha(z) \neq 0 \in \Br(k)$,
   alors $ A_{0}(X ) \to A_{0}(\X)^G \simeq \Alb_{X}(k)$ n'est pas surjectif.

   \end{prop}

\begin{proof}
Pour la courbe $C$, les fl\`eches
$A_{0}(C) \to \Alb_{C}(k)$
et $A_{0}(C) \to A_{0}({\overline C})^G$ sont injectives,
et l'\'enonc\'e analogue vaut sur tout corps contenant $k$.

On utilise la fonctorialit\'e, 
sur les
vari\'et\'es projectives,
de l'accouplement
entre le groupe de Chow des z\'ero-cycles et le groupe de Brauer.

Si $z=f_{*}(w)$, alors  $$ \alpha(z)= <z,\alpha>_{C}=<f_{*}(w),\alpha>_{C}= <w,f^*(\alpha)_{X}> = 0 \in \Br(k),$$
puisque $f^*(\alpha)=0 \in \Br(X)$. Donc $f_{*}: A_{0}(X) \to A_{0}(C)$ n'est pas surjectif.
La proposition \ref{propprincip} (c) donne le r\'esultat.
\end{proof}

 \begin{remark}   On peut formuler un \'enonc\'e sur une base $B$ de dimension quelconque.
 Soient  $B$ une $k$-vari\'et\'e projective, lisse, g\'eom\'e\-triquement connexe, et
 $f: X \to B$  un $B$-sch\'ema de Severi-Brauer.
 Soit $\alpha \in \Br(B)$ la classe associ\'ee.

  Soit  $F=k(B)$.
   Supposons    que  $A_{0}(B_{{\overline F}}) \to \Alb_{B}({\overline F})$
est un isomorphisme. 
 Supposons $X(k) \neq \emptyset$ et $\alpha \in \Br(B)$ non nul.
 Alors $A_{0}(X_{F}) \to A_{0}(B_{F})$
n'est pas surjectif comme on voit en appliquant $\alpha$ au point g\'en\'erique de $B$,
et on a au moins l'une des propri\'et\'es suivantes :

(i) $A_{0}(B_{F}) \to   \Alb_{B}(F)$ n'est pas injectif.

(ii) $A_{0}(X_{F}) \to  \Alb_{X}(F)$ n'est pas surjectif.
   \end{remark}

\begin{prop}\label{equivseveribrauer}
Soit $C$ une $k$-courbe projective, lisse, g\'eom\'e\-triquement connexe.
Soit $f : X \to C$ un sch\'ema de Severi-Brauer   dont la
classe associ\'ee  $\alpha \in \Br(C) \subset \Br(k(C))$ n'est pas nulle,
ce qui \'equivaut \`a dire que $f$ n'a pas de section.
Supposons $X(k) \neq \emptyset$.
Soit $F$ un corps contenant $k$.

Dans chacun des cas suivants, l'application 
 $ A_{0}(X_{F} ) \to A_{0}(\X_{F})^{G_{F}} \simeq \Alb_{X_{F}}(F)$
 n'est pas surjective.

(i) L'application $f_{F }: X(F) \to  C(F)$ n'est pas surjective.

(ii) L'application d'\'evaluation $$ev_{\alpha} : A_{0}(C_{F}) \to \Br(F)$$
n'est pas nulle.

(iii) On a  $F=k(C)$ et la classe $\alpha \in \Br(C)$  n'est pas dans l'image de $\Br(k)$.
\end{prop}

\begin{proof}  L'image de $f_{F }: X(F) \to  C(F)$ est exactement l'ensemble des points
 $P \in C(F)$  avec $\alpha(P) =0 \in \Br(F)$.
Si $X(k) \neq \emptyset$ et on a l'hypoth\`ese (i), alors $\alpha$ prend au moins
deux valeurs diff\'erentes sur $C(F)$, et (ii) est satisfait. On applique la proposition
\ref{severibcourbe} sur le corps $F$.
 Sous l'hypoth\`ese (iii), la classe $\alpha$ s'annule sur l'image d'un point de $X(k)$
 dans $C(k) \subset C(F)$
 et ne s'annule pas au point g\'en\'erique de $C$.  On applique la proposition
\ref{severibcourbe} sur le corps $F=k(C)$.
\end{proof}

\begin{cor} Soient $k$ un corps $p$-adique et $C/k$ une $k$-courbe connexe projective
et lisse de genre au moins 1 avec $C(k)\neq \emptyset$.
Il existe alors un sch\'ema de Severi-Brauer $X \to C$
tel que l'application
 $ A_{0}(X) \to A_{0}(\X)^{G} \simeq \Alb_{X}(k)$
 n'est pas surjective.
 \end{cor}

\begin{proof} Notons $J/k$ la jacobienne de la courbe $C$.
Si $k$ est un corps $p$-adique, on a la dualit\'e parfaite
de Tate 
$$J(k) \times H^1(k,J) \to \Q/\Z$$
(groupe compact, groupe discret).
D'apr\`es  Lichtenbaum \cite{Li69} ceci se r\'e\'ecrit
comme une dualit\'e parfaite
$$ A_{0}(C) \times \Br(C)/\Br(k) \to \Q/\Z.$$
Le groupe $A_{0}(C)=J(k)$  n'est pas nul, car ce groupe compact contient
un sous-groupe ouvert isomorphe \`a $O_{k}^g$, o\`u $O_{k}$ est l'anneau des entiers de $k$
et $g$ est le genre de la courbe $C$, qui est la dimension de $J$. La dualit\'e
donne donc l'existence d'une classe $\alpha\in \Br(C)$ qui ne s'annule
pas sur $A_{0}(C)$. On applique alors la proposition 
\ref{severibcourbe}.
\end{proof}

\begin{cor} Soit $C/\R$ une $\R$-courbe
connexe projective
et lisse telle que  l'espace topologique $C(\R)$ poss\`ede au moins deux composantes
connexes. 
Il existe alors un sch\'ema de Severi-Brauer $X \to C$
tel que l'application
 $ A_{0}(X) \to A_{0}(\X)^{G} \simeq \Alb_{X}(\R)$
 n'est pas surjective.
  \end{cor}
  
\begin{proof} 
   D'apr\`es Witt \cite[Behauptung II]{W34} pour $P, Q\in C(\R)$ dans deux composantes connexes distinctes,  
 il existe $\alpha \in \Br(C)$ tel que $\alpha(P)\neq \alpha(Q) \in \Br(\R)$.
 On applique la proposition \ref{severibcourbe}.
 \end{proof}

\begin{rmk}  
Par approximation, on peut donner des exemples sur
 un corps de nombres.  Plus simplement, soit $C/\Q$
 la courbe elliptique projective et lisse   des rationnels d\'efinie par l'\'equation affine
 $$y^2=x(x-1)(x+1).$$
 Soit $\Q(C)$ son corps des fonctions.
 La classe de l'alg\`ebre de quaternions $(x,-1) \in \Br(\Q(C))$ est une classe $\alpha\in \Br(C)$.
 Soit $X \to C$ un  sch\'ema de Severi-Brauer de classe $\alpha$. La classe $\alpha$  prend la valeur $(1,-1)=0 \in \Br(\Q)$ en le point $P \in C(\Q)$ 
 donn\'e par $(x,y)=(1,0)$
 et la valeur $(-1,-1) \neq 0 \in \Br(\Q)$ en le point $Q \in C(\Q)$ donn\'e par $(x,y)=(-1,0)$.
 On a  $\alpha(P) \neq \alpha(Q) \in \Br(\Q)$, comme on voit  d\'ej\`a dans $\Br(\R)=\Z/2$.
 Une application directe de la proposition \ref{equivseveribrauer} (ii) donne que l'application
 $A_{0}(X) \to \Alb_{X}(\Q)$ n'est pas surjective.
 \end{rmk}

 \begin{prop}
Soit $f: X \to Y$ un morphisme de $\R$-vari\'et\'es
 projectives, lisses,
\`a fibre g\'en\'erique   g\'eom\'etrique une vari\'et\'e rationnellement connexe.
Supposons $X(\R)\neq \emptyset$. Supposons que  l'application induite sur les composantes
connexes $\pi_{0}(X(\R)) \to \pi_{0}(Y(\R))$ n'est pas surjective. Alors :

(i) L'application $f_{*} : A_{0}(X)/2 \to A_{0}(Y)/2$ n'est pas surjective.

(ii) Si $Y$ est une courbe
alors l'application $A_{0}(X) \to \Alb_{X}(\R)$ n'est pas surjective.
 \end{prop}
\begin{proof} D'apr\`es \cite{CTI81}, on a un isomorphisme naturel
 $$CH_{0}(X)/2 \simeq (\Z/2)^{\pi_{0}(X)}.$$ Ceci \'etablit (i).
 L'\'enonc\'e (ii) r\'esulte alors de la proposition \ref{propprincip} (c).
 \end{proof}

\begin{remark}
Lorsque $k$ est un corps fini $\F$, la th\'eorie du corps de classes sup\'erieur
(K.~Kato et S. Saito) montre que pour toute $\F$-vari\'et\'e projective et lisse
g\'eom\'e\-tri\-quement connexe $X$, l'application
$A_{0}(X) \to \Alb_{X}(\F)$ est surjective, l'application
$A_{0}(\X)  \to \Alb_{X}({\overline {\F}})$ est   un isomorphisme,
et l'application $A_{0}(X)\to A_{0}(\X)^{G_{\F}}$ est surjective.

 Lorsque $k$ est un corps de nombres, 
 Bloch et Beilinson ont conjectur\'e 
que l'application $A_{0}(\X)  \to \Alb_{X}({\overline k})$ est un isomorphisme :
 c'est une cons\'equence facile de  \cite[p.~121, Conjecture]{Bl84}.
 
\end{remark}

\section{Vari\'et\'es complexes avec un z\'ero-cycle universel}

La notion de z\'ero-cycle universel pour une vari\'et\'e projective
et lisse sur les complexes est introduite par C. Voisin dans  \cite{V24a} et \cite{V24b}.
 On traduit certains des \'enonc\'es  de \cite{V24a}
 dans le langage des groupes de Chow de z\'ero-cycles sur des corps arbitraires.

Soient  $X /{\mathbb C} $  une vari\'et\'e projective et lisse, $m \in X({\mathbb C})$
et $\Alb_{X}/{\mathbb C}$ la vari\'et\'e ab\'elienne d'Albanese.
Soit $f : X \to \Alb_{X}$ le morphisme  d'Albanese
envoyant $m$ sur $0 \in \Alb_{X}({\mathbb C})$.
Comme rappel\'e \`a la proposition \ref{roitman}, ceci induit un homomorphisme surjectif
$$A_{0}(X) \to \Alb_{X}(\C)$$ 
induisant un isomorphisme sur les groupes de torsion 
 (Roitman \cite{R80}).

\begin{thm}\label{representable} (Roitman \cite{R72})
Avec les notations ci-dessus, les conditions  suivantes sont \'equi\-valentes :

(a)  Il existe un entier $d>0$ tel que l'application 
$$X(\C)^d \times X(\C)^d \to  A_{0}(X)$$
envoyant $(x_{1}, \dots, x_{d}); (y_{1}, \dots, y_{d})$ sur la classe
de $x_{1}+ \dots +x_{d} -y_{1} - \dots - y_{d}$ 
soit surjective.

(b)
La fl\`eche surjective  
$$ A_{0}(X) \to  \Alb_{X}({\mathbb C})$$
est un isomorphisme.

(c)
  Il existe une courbe $\Gamma/\C$ connexe, projective et lisse et un morphisme $\Gamma \to X$
tels que l'application induite $A_{0}(\Gamma) \to A_{0}(X)$ est surjective.

(d) Pour tout corps alg\'ebriquement clos $\Omega$
contenant  ${\mathbb C} $, la fl\`eche surjective 
$$ A_{0}(X_{\Omega}) \to  \Alb_{X}(\Omega)$$
est un isomorphisme.  
\end{thm}

\begin{proof}

Que  (a) implique (b) est  le th\'eor\`eme \cite[\S 4, Thm. 4]{R72} de Roitman  (voir aussi \cite[\S 22.1.2]{V02}). 

Supposons (b).
Soit $\Gamma \subset X $ une courbe lisse  intersection compl\`ete  de sections hyperplanes de $X$.
On a alors un isomorphisme $A_{0}(\Gamma) \simeq \Alb_{\Gamma}(\C)$ et, par un th\'eor\`eme de Lefschetz, une surjection
 $\Alb_{\Gamma}(\C) \to \Alb_{X}(\C)$. 
 Si $ A_{0}(X) \to  \Alb_{X}({\mathbb C})$
 est un isomorphisme, alors $A_{0}(\Gamma)  \to A_{0}(X)$ est surjectif.
 Ainsi (b) implique (c).
 
 Que (c) implique (a) r\'esulte imm\'ediatement  du cas $X=\Gamma$
pour lequel l'\'enonc\'e suit  
 du th\'eor\`eme de Riemann-Roch sur une courbe.

L'\'enonc\'e (d) g\'en\'eralise (b). La $\C$-vari\'et\'e $X$ s'\'ecrit   $X=X_{0}\times_{k_{0}}\C$ avec
$k_{0} \subset \C$ un corps de type fini sur $\Q$.  Pour $F$ variant parmi les corps extensions de
$k_{0}$, chacun des foncteurs
$F \mapsto CH_{0}(X_{0}\times_{k_{0}}F)$ et
$F \mapsto \Alb_{X_{0}}(F)$ commute aux limites inductives filtrantes,
et pour $F \subset F'$ les applications $Alb_{X_{0}}(F) \to \Alb_{X_{0}}(F')$
sont injectives.
On en d\'eduit que (b) est \'equivalent \`a (d).
\end{proof}

Si les conditions \'equivalentes du th\'eor\`eme  \ref{representable}  sont satisfaites, on dit (classiquement) que
 {\it le groupe de Chow des z\'ero-cycles de la vari\'et\'e projective et lisse $X/\C$   est repr\'esentable}.
 
 \begin{lemma}\label{pasweil} 
 (a) Soit $X/\C$ une vari\'et\'e int\`egre. Pour tout corps $F$ contenant $\C$,
$X_{F}(F)$ est dense dans $X_{F}$ pour la topologie de Zariski.

(b) Soit $A/\C$ un groupe alg\'ebrique connexe.
Soit  $F$ un corps contenant $\C$.   Soit $U \subset A_{F}$
un ouvert de Zariski non vide. Tout \'el\'ement   $x \in A(F)$
s'\'ecrit $x=a.b^{-1}$ avec  $a, b \in U(F)$.
 \end{lemma}
 
 \begin{proof}
(a) Il suffit de montrer que $X(\C) \subset X(F)$ est Zariski dense dans $X_{F}$.
Si ce n'est pas le cas, il existe un ferm\'e strict $Y \subset X_{F}$
contenant $X(\C)$. Comme ce ferm\'e est d\'efini par un nombre fini
d'\'equations,  le m\^eme \'enonc\'e vaut avec $F/\C$ le corps des fractions
d'une $\C$-alg\`ebre  $A$ de type fini. Quitte \`a inverser un nombre fini d'\'el\'ements dans $A$,
on peut supposer qu'il
 existe un sous-$A$-sch\'ema ferm\'e  ${\mathcal Y} \subset X\times_{\C} A$
qui est strictement contenu dans $X_{A}:=X\times_{\C} A$ qui par passage de $A$ \`a $F$
donne l'inclusion $Y\subset X_{F}$.
Quitte \`a restreindre $A$, on peut supposer que  pour tout $\C$-point $m$ de $Spec(A)$, 
c'est-\`a-dire tout $\C$-homorphisme $A \to \C$,
l'inclusion induite ${\mathcal Y}_{m} \subset X\times_{A}\C$ est stricte.
Pour tout point $m$, l'inclusion $X(\C) \subset  Y(F) \subset X_{F}(F) $ induit une inclusion
 $X(\C) \subset   {\mathcal Y}_{m}(\C) \subset X(\C)$ 
la fl\`eche compos\'ee \'etant l'identit\'e. Contradiction.

(b) Soit $x\in A(F)$.  Il existe un point de $A(F)$ dans l'ouvert $x.U \cap U$ de $A_{F}$.  On peut donc \'ecrire
$x.b = a \in A(F)$ avec $a, b \in U(F)$.
 \end{proof}

\begin{prop}\label{zerocycleuniv}
Soit  $X /{\mathbb C} $  une vari\'et\'e projective et lisse connexe.
 Consid\'erons les \'enonc\'es :

(i) Pour tout corps $F/{\mathbb C}$, l'application
$A_{0}(X_{F}) \to  A_{0}(X_{\overline F})^{G_{F}}$
est surjective.

(ii)  Pour tout corps $F/{\mathbb C}$, l'application
$A_{0}(X_{F}) \to  \Alb_{X}(F)$
est surjective.

(iii) Soit $F={\mathbb C}(\Alb_{X})$. L'application
$A_{0}(X_{F}) \to  \Alb_{X}(F)$
a le point g\'en\'erique de $\Alb_{X}$ dans son image.

La condition (i) implique (ii), et la condition (ii) est
\'equivalente \`a (iii).

Si  le groupe de Chow des z\'ero-cycles de $X$ est repr\'esentable,
alors les  trois  propri\'et\'es  sont \'equivalentes.
\end{prop}

\begin{proof}  Comme cons\'equence du th\'eor\`eme de Roitman sur
la torsion du groupe de Chow, on a vu que l'application
$A_{0}(X_{\overline F})^{G_{F}}  \to \Alb_{X}(F)$ est surjective.
Donc (i) implique (ii). Que (ii) implique (iii) est \'evident.

Montrons que (iii) implique (ii).
Fixons un point $P \in X(\C)$. Soit $E=\C(\Alb_{X})$.
L'hypoth\`ese  (iii) implique l'existence d'un entier $n\geq 1$
tel que le  $\C$-morphisme
$$ \theta_{n } : \Sym^n(X/\C) \to   \Alb_{X}$$
 associant \`a un $n$-cycle effectif $(z_{1}+ \dots +z_{n})$
 l'image de $z_{1}+\dots+z_{n} - nP $  dans $\Alb_{X}$
 satisfasse : il existe un point de $\Sym^n(X/\C)(E)$ 
 d'image le point g\'en\'erique de $\Alb_{X}$. Ainsi, il existe
 un ouvert non vide $U \subset \Alb_{X}$ et un morphisme
 $\rho : U \to \Sym^n(X)$  tel que $\theta_{n} \circ \rho$
 soit l'inclusion de $U$ dans $\Alb_{X}$.
Alors, pour tout corps $F/\C$, l'ensemble
$U(F) \subset \Alb_{X}(F)$ est dans l'image de 
$A_{0}(X_{F})$. 
Comme d'apr\`es le lemme \ref{pasweil} l'application
$$U(F) \times U(F) \to \Alb_{X}(F)$$ donn\'ee par la soustraction
est surjective, 
ceci donne le r\'esultat.

Montrons la derni\`ere assertion. 
On a le diagramme commutatif
 \[
\begin{tikzcd}
A_{0}(X_{F} \arrow[d])   \arrow[r] & \Alb_{X}(F)\arrow[d] \\
A_{0}(X_{\overline{F}})^G      \arrow[r]  & \Alb_{X}({\overline{F}})^G 
\end{tikzcd}
	\]	

Sous l'hypoth\`ese (ii), la fl\`eche horizontale sup\'erieure est surjective.
La repr\'esen\-tabilit\'e sous la forme (d) dans le th\'eor\`eme \ref{representable}
assure que la fl\`eche horizontale inf\'erieure est un isomorphisme.
 Comme la fl\`eche verticale
de droite est un isomorphisme, on conclut que la fl\`eche
$A_{0}(X_{F}) \to  A_{0}(X_{\overline F})^{G_{F}}$
est surjective.
\end{proof}

On donne \`a la propri\'et\'e (ii) de la proposition \ref{zerocycleuniv} un nom :
 \begin{defin}
  (Voisin  \cite{V24a, V24b}) 
 {\it  Soit $X/\C$ une vari\'et\'e connexe, projective et lisse.
 Si pour tout corps $F/\C$, l'application $A_{0}(X_{F}) \to \Alb_{X}(F)$ est surjective,
 on dit que  la vari\'et\'e $X$
poss\`ede un z\'ero-cycle universel param\'etr\'e par $\Alb_{X}$.}
  \end{defin}

\begin{prop}  Si $X$ est une courbe, ou si $X$ est une vari\'et\'e ab\'elienne,
alors $X$
poss\`ede un z\'ero-cycle universel param\'etr\'e par $\Alb_{X}$.
\end{prop}
\begin{proof}
Soit $X=C$ est une courbe. Dans ce cas, pour tout corps $F/{\mathbb C}$,
les fl\`eches $$A_{0}(C_{F}) \to  A_{0}(C_{\overline F})^{G_{F}} \to \Alb_{C}(F)$$
sont des isomorphismes.
 
Soit $X=A/\C$ est une vari\'et\'e ab\'elienne, alors
 on a $A=Alb_{A}$, et tout point $m\in  \Alb_{A}(F)$
est l'image du z\'ero-cycle $[m] - [0] \in Z_{0}(A_{F})$.
\end{proof}

Notons que  si $X=A$ est une vari\'et\'e ab\'elienne avec $dim(A)\geq 2$,
 le groupe de Chow des z\'ero-cycles n'est pas repr\'esentable: 
la fl\`eche $A_{0}(A) \to \Alb_{A}({\mathbb C})=A(\C)$ n'est pas un isomorphisme 
(cas particulier du th\'eor\`eme de Mumford).

Voisin  \cite[Cor. 0.14, Cor. 3.1]{V24b} donne des exemples de vari\'et\'es $M/\C$, et m\^{e}me de surfaces
$M$,  pour lesquelles il n'existe pas de z\'ero-cycle universel param\'etr\'e
par $\Alb_{M}$.
Pour une telle vari\'et\'e $M$,   pour $F={\mathbb C}(\Alb_{M})$, la fl\`eche  $A_{0}(M_{F}) \to \Alb_{M}(F)$ n'est pas surjective,
et  la fl\`eche $A_{0}(M_{F}) \to  A_{0}(M_{\overline F})^{G_{F}}$
n'est pas surjective.

 \begin{question}
 Peut-on donner un exemple de vari\'et\'e projective et lisse  $M$ sur $\mathbb C$
dont le groupe de Chow  des z\'ero-cycles de $M$ est repr\'esentable, c'est-\`a-dire que
$$alb : A_{0}(M) \to \Alb_{M}({\mathbb C})$$
est un isomorphisme
(et donc   $H^{i}(M,O_{M} )=0$  pour $i \geq 2$),
mais il existe un corps $F/{\mathbb C}$ avec
$$A_{0}(M_{F}) \to \Alb_{M}(F)$$
non surjectif ?
\end{question}

\begin{rmk}
Soit $X/\C$ une surface d'Enriques. Elle est munie d'un rev\^etement double
non ramifi\'e $Y \to X$, avec $Y$ int\`egre. Ceci d\'efinit une classe non nulle  $\xi \in H^1_{et}(X,\Z/2)$.
Bloch, Kas et Liebermann ont montr\'e $A_{0}(X)=0$. Soit $F=\C(X)$. Soit $\eta \in X(F)$ le point g\'en\'erique de
$X$ et $m\in X(\C) \subset X(F)$. L'image de $(\eta - m, \xi)$ par l'accouplement bilin\'eaire
$$ A_{0}(X_{F}) \times H^1_{et}(X,\Z/2) \to H^1(F, \Z/2)$$
est la classe de l'image de $\xi$ dans $H^1_{\text{\'et}}(\C(X), \Z/2)$, qui n'est pas nulle, car $Y \to X$
n'a pas de section rationnelle.
Donc $A_{0}(X_{F})\neq 0$.
 \end{rmk}

Pour $X \to Y$ 
un sch\'ema de Severi-Brauer, et plus g\'en\'eralement pour
$X \to Y$ un morphisme dominant \`a fibre g\'en\'erique
g\'eom\'etrique rationnellement connexe,
on a un isomorphisme $\Alb_{X} \simeq \Alb_{Y}$
(lemme \ref{isoalb}). 
 Pour $Y$ poss\'edant un z\'ero-cycle universel  param\'etr\'e par $\Alb_{Y}$,
dans \cite{V24a}, Voisin  examine la question s'il
 existe pour $X$  un  z\'ero-cycle universel  param\'etr\'e
par $\Alb_{X} $.

\begin{lemma}\label{produit}
Soit $X = Y \times_{\mathbb C}Z$ un produit 
de
vari\'et\'es projectives et lisses sur $\mathbb C$.
Pour tout corps $F/{\mathbb C}$, l'homomorphisme
$$A_{0}(X_{F}) \to  \Alb_{X}(F)$$
est surjectif si et seulement si il
l'est pour $Y$ et pour  $Z$.
\end{lemma}
\begin{proof}  
L'isomorphisme de vari\'et\'es ab\'eliennes 
$$\Alb_{X} \times \Alb_{Y} \simeq \Alb_{Z}$$
 induit un isomorphisme de groupes ab\'eliens
$$ \Alb_{X}(F) \times \Alb_{Y}(F)  \simeq \Alb_{Z}(F).$$

La fonctorialit\'e covariante en $X$ de l'application  $A_{0}(X_{F}) \to  \Alb_{X}(F)$
donne que la surjectivit\'e  de l'application d'Albanese sur $F$ pour les z\'ero-cycles 
pour $X$ l'implique pour $Y$ et pour $Z$.

Le produit ext\'erieur des cycles donne
une application 
$$A_{0}(Y_{F}) \times A_{0}(Z_{F}) \to A_{0}(X_{F}).$$
Le diagramme
   \[
\begin{tikzcd}
 A_{0}(Y_{F})\arrow[d]  &  \times   & A_{0}(Z_{F}) \arrow[r] \arrow[d] &  A_{0}(X_{F}) \arrow[d] \\
\Alb_{Y}(F)  &  \times  &  \Alb_{Z} (F)  \arrow[r] &   \Alb_{X}(F)
 \end{tikzcd}
	\]	
est fonctoriel en le corps $F$.  On v\'erifie qu'il est  commutatif
en passant \`a  une cl\^{o}ture alg\'ebrique de $F$.
On en d\'eduit que la surjectivit\'e de l'application d'Albanese sur $F$ 
pour les z\'ero-cycles  pour $Y$ et pour $Z$
l'implique pour $X$.
\end{proof}

\begin{lemma}\label{birationnel}
Si $X$ et $Y$ sont deux vari\'et\'es projectives, lisses sur $\C$,
stablement birationnelles entre elles, $X$ admet un z\'ero-cycle universel  sur $X$ param\'etr\'e par
$\Alb_{X}$ si et seulement si il en est de m\^{e}me de $Y$.
\end{lemma}
\begin{proof} Un  morphisme $p: Z\to X$   de $\C$-vari\'et\'es projectives et lisses
g\'eom\'etriquement int\`egres induit  pour tout corps $F/\C$  un carr\'e commutatif
  \[
\begin{tikzcd}
  A_{0}(Z_{F})   \arrow[r] \arrow[d] &      A_{0}(X_{F})  \arrow[d] \\
 \Alb_{Z}(F)   \arrow[r] &  \Alb_{X}(F)
 \end{tikzcd}
	\]	
Pour  $Z= \P^n \times_{\C}X$ et la projection sur $X$, et lorsque $Z \to X$ est un  morphisme birationnel,
 on sait que les fl\`eches horizontales sont des isomorphismes. 
Par r\'esolution des singularit\'es on est ramen\'e \`a ces deux cas.
\end{proof}

Soit $\mathcal{C}$ la classe des vari\'et\'es $Y/\C$ connexes, projectives
et lisses satisfaisant la propri\'et\'e :

{\it  Pour tout morphisme dominant $X \to Y$ de vari\'et\'es
connexes, projectives et lisses \`a fibre g\'en\'erique g\'eom\'etrique
une vari\'et\'e rationnellement connexe, et pour tout corps $F/\C$,
l'application $A_{0}(X_{F}) \to \Alb_{X}(F)$ est surjective.}

Cette propri\'et\'e implique en particulier sa validit\'e pour $Y=X$.

Pour un morphisme comme ci-desssus, d'apr\`es le lemme \ref{isoalb}, la fl\`eche
 $\Alb_{X} \to \Alb_{Y}$ est un isomorphisme. Ceci est utilis\'e tacitement
 dans les d\'emonstrations suivantes.

\begin{lemma}\label{stabilite}
La classe $\mathcal{C}$ est stable par \'equivalence birationnelle, 
par produit, et par facteur direct birationnel.
\end{lemma}
\begin{proof}
(a) La stabilit\'e par  \'equivalence birationnelle r\'esulte du lemme \ref{birationnel}.

(b) Rappelons que la propri\'et\'e pour une vari\'et\'e projective et lisse d'\^{e}tre
g\'eom\'e\-tri\-quement rationnellement connexe est stable par sp\'ecialisation
(Koll\'ar-Miyaoka-Mori).
Soit $p= X \to Y=Y_{1}\times Y_{2}$ un morphisme projectif \`a fibre g\'en\'erique g\'eom\'etrique
une vari\'et\'e rationnellement connexe. Soit $m=(m_{1},m_{2})$
un point de $Y(\C)$ \`a fibre  lisse. 
Soit $p_{1} : W_{1} \to Y_{1}$ la restriction de $p$ au-dessus de $Y_{1} \times m_{2}$
et $p_{2} : W_{2} \to Y_{2}$ la restriction de $p$ au-dessus de $m_{1} \times Y_{2}$.
Les fibres g\'en\'eriques  g\'eom\'etriques de $p_{1}$ et de $p_{2}$ sont lisses et rationnellement connexes.
Soit $\tilde{W_{i}} \to W_{i}$ une d\'esingularisation de l'adh\'erence dans $W_{i}$  de  la fibre g\'en\'erique de $p_{i}$.
 L'assertion pour le produit r\'esulte alors du fait que  l'on a 
$$\Alb_{Y}(F) =\Alb_{Y_{1}}(F) \times \Alb_{Y_{2}}(F)$$
et de la fonctorialit\'e  en $X$ de l'application $A_{0}(X_{F}) \to \Alb_{X}(F)$
appliqu\'ee aux morphismes $\tilde{W}_{i} \to X$.

(c)  Comme on l'invariance birationnelle,  il suffit de montrer que l'assertion : si la propri\'et\'e vaut pour
$Y_{1} \times Y_{2}$, elle vaut pour $Y_{1}$.
Soit $X_{1} \to Y_{1}$ un morphisme dominant de vari\'et\'es 
connexes, projective lisses  \`a fibre g\'en\'erique g\'eom\'etrique
rationnellement connex. Le morphisme $X:=X_{1} \times Y_{2} \to 
Y:= Y_{1} \times Y_{2}$ donn\'e par l'identit\'e sur le second facteur
a la m\^eme propri\'et\'e. La surjectivit\'e de $A_{0}(X_{F}) \to \Alb_{X}(F)=\Alb_{Y}(F)$
implique celle de   $A_{0}(X_{1,F}) \to \Alb_{X_{1}}(F)= \Alb_{Y_{1}}(F)$.
\end{proof}

La classe $\mathcal{C}$ contient les courbes :
\begin{prop}\label{courbe}
Soit $C$ une courbe projective et lisse.
Soit $X \to C$ un morphisme dominant \`a fibre g\'en\'erique
g\'eom\'etrique rationnellement connexe.
Pour tout corps $F/{\mathbb C}$, l'application
$A_{0}(X_{F}) \to  \Alb_{X}(F)$ est surjective.
\end{prop}
\begin{proof}
D'apr\`es le th\'eor\`eme de Graber, Harris et Starr, 
le morphisme $X \to C$ admet une section.
Ceci implique que $A_{0}(X_{F}) \to A_{0}(C_{F})=\Alb_{C}(F)$ est surjectif,
et donc aussi  $A_{0}(X_{F}) \to  \Alb_{X}(F)$.
\end{proof}

La classe $\mathcal C$ contient les jacobiennes de courbes :
\begin{prop}\label{jacobienne}
Soit $A=J(C)$  la jacobienne d'une courbe projective et lisse $C$.
Soit $X \to A$ un morphisme dominant \`a fibre g\'en\'erique
g\'eom\'etrique rationnellement connexe.
Pour tout corps $F/{\mathbb C}$, l'application
$$A_{0}(X_{F}) \to  \Alb_{X}(F)$$
est surjective.
\end{prop}
\begin{proof}
Le cas $C= \P^1$ est clair. Supposons $C$ de genre au moins 1.
Par translation, on peut supposer que
l'on a un plongement $C \hookrightarrow J(C)=A$ tel que la restriction
de $X \to A$ au-dessus du point g\'en\'erique de $C$ est une vari\'et\'e projective et lisse
g\'eom\'etriquement int\`egre rationnellement connexe. 
Par le th\'eor\`eme de Graber-Harris-Starr (ou par le th\'eor\`eme de Tsen
si $X \to A$ est un sch\'ema de Severi-Brauer), 
 la projection $Y:= X_{C} \to C$
admet une section rationnelle. Comme $C$ est une courbe lisse et le morphisme
$Y \to C$ est propre, toute telle section rationnelle est un morphisme.

Pour tout corps $F/{\mathbb C}$, on a le diagramme commutatif suivant
  \[
\begin{tikzcd}
A_{0}(Y_{F}) \arrow[r] \arrow[d] &  A_{0}(X_{F}) \arrow[r] \arrow[d] & \Alb_{X}(F) \arrow[d] \\
A_{0}(C)(F) \arrow[r] &  A_{0}(A_{F}) \arrow[r] &  \Alb_{A}(F)
 \end{tikzcd}
	\]	
La fl\`eche compos\'ee $A_{0}(C)(F) \to  A_{0}(A_{F}) \to  \Alb_{A}(F)= \Alb_{C}(F)$
est un isomorphisme. La fl\`eche $\Alb_{X} \to \Alb_{A}$ est un isomorphisme (lemme \ref{isoalb}). 
La fl\`eche $A_{0}(Y_{F}) \to A_{0}(C)(F)$ est surjective puisque le morphisme $Y \to C$ admet une section. On conclut que
la fl\`eche $A_{0}(X_{F}) \to \Alb_{X}(F)$ est   surjective.
\end{proof}

\begin {rmk}
 Le point cl\'e dans la d\'emonstration est le m\^{e}me que  dans \cite[Prop. 1.7]{V24a}.
c'est le th\'eor\`eme de Graber, Harris et Starr.  
La question si la proposition ci-dessus
vaut pour $A$ une vari\'et\'e ab\'elienne quelconque est ouverte, d\'ej\`a dans le cas
o\`u $X\to A$ est un sch\'ema de Severi-Brauer   \cite{V24a}.
\end{rmk}

En combinant  le lemme \ref{produit}, le lemme \ref{stabilite}, la proposition \ref{courbe} 
 et la proposition  \ref{jacobienne},
on obtient  la proposition suivante, l\'eg\`ere g\'en\'eralisation de  \cite[Prop. 1.7]{V24a}.
\begin{prop}
Si une vari\'et\'e projective et lisse $Y$ est facteur  direct  birationnel
d'un produit de  courbes et  de jacobiennes de courbes,
et si $f: X \to Y$ est un morphisme projectif dominant de vari\'et\'es projectives 
lisses \`a fibre g\'en\'erique g\'eom\'etrique
  rationnellement connexe,
alors la vari\'et\'e 
$X$ poss\`ede un z\'ero-cycle universel param\'etr\'e par
$\Alb_{X} \simeq \Alb_{Y}$.
\end{prop}

\begin{rmk} La question 
si toute vari\'et\'e ab\'elienne
 est facteur direct, comme vari\'et\'e ab\'elienne, d'un produit de jacobiennes est un probl\`eme ouvert,
 discut\'e dans \cite{V24a}.  
\end{rmk}

\section{Solides rationnellement connexes sur les complexes}

 On donne ici des variations sur le th\`eme de  la section 3  de l'article
 \cite{V24b} de Voisin.
 
 Soit $X/\C$ une vari\'et\'e connexe, projective et lisse.
  Pour tout corps $F/\C$ on note $A^2(X_{F}):=CH^2(X_{F})_{alg} \subset CH^2(X_{F})$
 le sous-groupe form\'e des classes de cycles de codimension 2 dont l'image dans $CH^2(X_{\overline F})$
 est alg\'ebriquement \'equivalente \`a z\'ero.
  \'Etant donn\'e une vari\'et\'e connexe, projective et lisse $M/\C$ et un cycle $Z$  de codimension 2 sur $M \times X$,
 pour tout corps $F/\C$, on a une application induite
 $$\Theta_{Z} : CH_{0}(M_{F}) \to CH^2(X_{F})$$ et donc une application
 $A_{0}(M_{F}) \to  A^2(X_{F})$ fonctorielle en $F/\C$.
 On note ici $A_{0}(X_{F}) \subset CH_{0}(X_{F})$ le sous-groupe des classes de z\'ero-cycles
 de degr\'e z\'ero.
 On a par ailleurs l'homomorphisme 
 $A_{0}(M_{F}) \to \Alb_{M}(F)$, qui est fonctoriel en $F/\C$.
 
 G\'en\'eralisant des travaux de  Murre, Bloch, Srinivas,    
 Voisin  \cite[Cor. 0.9]{V24b} montre }\footnote{Pour l'aspect fonctoriel en le corps $F$,  voir les travaux de Achter, Casalaina-Martin et Vial
  \cite{ACMV23}.} :
 \begin{thm} \label{solidesRC}
Soit $X/\C$ un solide projectif et lisse rationnellement connexe.
 Il  existe une surface projective et lisse  $S$
   et
un cycle $Z$ de codimension 2 sur $S \times X$ qui,  pour tout corps $F/\C$,
induit un homomorphisme $A_{0}(S_{\overline{F}}) \to A^2(X_{\overline{F}}   )$ 
Galois-\'equivariant qui 
  se factorise par l'application d'Albanese de $S$ :$$ A_{0}(S_{\overline{F}}) \to \Alb_{S}(\overline{F}) \to  A^2(X_{\overline{F}}),$$ 
   l'application $Alb_{S}(\overline{F}) \to  A^2(X_{\overline{F}})$ \'etant 
    un isomorphisme Galois \'equivariant.
\end{thm}
  On a donc  un diagramme commutatif
 $$\xymatrix{
 A_{0} (S_{\overline F}) \ar[r]   &   \Alb_{S}({\overline F})  \ar[r]^{\simeq}  & A^2(X_{\overline F})  \\
 A_{0} (S_{\overline F})^G   \ar[r] \ar[u] &   \Alb_{S}({\overline F})^G \ar[u] \ar[r]^{\simeq} &    A^2(X_{\overline F})^G \ar[u] \\
 A_{0}(S_{F}) \ar[u]\ar[r] & \Alb_{S}(F)  \ar[u]^{\simeq}  &  A^2(X_{F}) \ar[u]\\
 A_{0}(S_{F}) \ar[u]^{=} \ar[rr] &  &  A^2(X_{F}) \ar[u]^{=}
 }
 $$	
  
  Via  l'inverse de l'isomorphisme $Alb_{S}({\overline F})\simeq    A^2(X_{\overline F})$,
 ce diagramme induit un homomorphisme 
 $\theta : A^2(X_{F}) \to \Alb_{X}(F)$, et  le compos\'e
 $A_{0}(S_{F}) \to A^2(X_{F}) \to \Alb_{X}(F)$ 
 est \'egal \`a la fl\`eche $A_{0}(S_{F}) \to \Alb_{X}(F)$,
 comme on voit en composant avec l'injection de
 $  \Alb_{S}(F)$ dans $ A^2(X_{\overline F})^G$.

 \begin{thm}\label{equivgenerales}
 Soit $X/\C$ un solide projectif et lisse rationnellement connexe.
Soient $S/\C$ et $Z/\C$ comme ci-dessus. Supposons que l'on a $\Br(X)=0$.
Soit $F/\C$ un corps.

L'application $A^2(X_{F}) \to A^2(X_{\overline F})$ est injective.

Consid\'erons les hypoth\`eses suivantes.
 
 A) L'application $ A_{0}(S_{F}) \to \Alb_{S}(F)$ est surjective.
  
   (B) L'application $\theta : A^2(X_{F})   \to  \Alb_{S}(F)$ est un isomorphisme.

(C) L'application 
$A^2(X_{F})  \to  A^2(X_{\overline F})^G$  est un isomorphisme.

(D) L'application  
$ CH^2(X_{F}) \to  CH^2(X_{\overline F})^G$  est un isomorphisme.

(E) On a  $H^3(F,\Q/\Z(2))= H^3_{nr}(F(X)/F,\Q/\Z(2) )$.

\medskip

On a les implications suivantes

$\bullet$  (A) implique (B)

$\bullet$  (B), (C) et (D) sont \'equivalents

$\bullet$  Chacune des hypoth\`eses pr\'ec\'edentes implique (E).
  
  \end{thm}
  
  \begin{proof}
  Pour $X/\C$ rationnellement connexe, le module galoisien
  $\Pic( X_{{\overline F}})$ est un r\'eseau avec action triviale
  de $G_{F}$. Si de plus $X$ est de dimension 3, alors
  $H^3_{nr}(X_{\overline{F}},\Q/\Z(2))=0$
  \cite[Cor. 6.2]{CTV12}.
  Pour $X$ rationnellement connexe de dimension 3 avec
  $\Br(X)=0$, le corollaire 4.2 (iii)  de \cite{CT15} donne donc  d'une part  
  une injection
  $$ CH^2(X_{F}) \hookrightarrow  CH^2(X_{{\overline F}}),$$
  et donc une injection $$ A^2(X_{F}) \hookrightarrow  A^2(X_{{\overline F}})$$
  d'autre part une injection
    $$ H^3_{nr}(F(X)/F,\Q/\Z(2) )/ H^3(F,\Q/\Z(2)) \hookrightarrow 
    \Coker[CH^2(X_{F})\to  CH^2(X_{{\overline F}})^G].$$
   On a les suites exactes compatibles
    $$0 \to A^2(X) \to CH^2(X) \to \NS^2(X) \to 0$$
  $$0 \to A^2(X_{F}) \to CH^2(X_{F}) \to \NS^2(X_{F}) \to 0$$
     $$0 \to A^2(X_{{\overline F}}) \to CH^2(X_{{\overline F}}) \to \NS^2(X_{{\overline F}}) \to 0.$$
   Pour $X$ comme ci-dessus,  l'application $\NS^2(X) \to \NS^2(X_{{\overline F}})$ est un isomorphisme
     \cite[Prop. 5.1 (iv)]{CT15},
      donc l'application $\NS^2(X_{F}  ) \to \NS^2(X_{{\overline F}})^{G_{F}}$ est surjective.
     Ainsi les conoyaux  de
      $$A^2(X_{F}) \to A^2(X_{{\overline F}})^G$$
      et de 
     $$CH^2(X_{F}) \to CH^2(X_{{\overline F}})^G$$
co\"{\i}ncident.

La d\'emonstration des implications du th\'eor\`eme suit alors de la consid\'eration du diagramme
suivant le th\'eor\`eme  \ref{solidesRC}.
  \end{proof}
  
  \begin{rmk} 
  Les hypoth\`eses ``(C) pour tout $F/\C$''  ou ``(D) pour tout $F/\C$'' sont des versions de l'hypoth\`ese :
 `` La vari\'et\'e $X$ poss\`ede un cycle de codimension deux universel''
 dans le contexte de Voisin \cite{V15} (voir \cite[\S 5.2]{CT15}.  L'\'equivalence des conditions
 ``(B) pour tout $F/\C$''  et ``(E) pour tout $F/\C$'' dans ce cadre est \cite[Cor. 3.3]{V15}). 
 Des \'enonc\'es dans un cadre plus g\'en\'eral sont  \cite[Thm. 1.10, Thm. 3.1]{V15} et \cite[Thm. 5.4]{CT15}.
  \end{rmk}
 
 \begin{rmk}
 Comme \'etabli dans \cite[Thm. 1.9, Cor. 1.11]{V15}, il existe des  solides rationnellement connexes $X$   avec $\Br(X)=0$
 pour lesquels il existe un corps $F$ tel qu'aucune des propri\'et\'es (A), (B), (C), (D), (E)  ne valent.
 Cette remarque est une variante de   \cite[Cor. 0.14, Cor. 3.1]{V24b}. 
 \end{rmk}

\begin{rmk}
Par un th\'eor\`eme r\'ecent de Koll\'ar et Tian \cite[Thm. 6]{KT24} sur les 1-cycles, 
pour  tout solide projectif et lisse rationnellement connexe $X$, et tout corps $F/\C$,  
l'application $A^2(X_{F}) \to A^2(X_{\overline F})$ est injective, et il en est
alors de m\^{e}me de l'application
 $CH^2(X_{F}) \to CH^2(X_{\overline F})$.
Ceci implique que la fl\`eche  
$A_{0}(S_{F}) \to A^2(X_{F})$ dans 
le grand diagramme pr\'ec\'edant le th\'eor\`eme \ref{equivgenerales} se factorise 
via l'homomorphisme
$A_{0}(S_{F}) \to \Alb_{S}(F)$, et  que
 la fl\`eche $\theta : A^2(X_{F}) \to \Alb_{X}(F)$
 est  injective.  
 Comme le note Z.~Tian, l'injectivit\'e de  $A^2(X_{F}) \to A^2(X_{\overline F})$
assure que les conditions (B), (C), (D) ci-dessus sont \'equivalentes pour tout  tel solide 
projectif et lisse rationnellement connexe $X$, 
 sans qu'on ait besoin de  l'hypoth\`ese $\Br(X)=0$ faite au th\'eor\`eme \ref{equivgenerales}.

Ceci dit, pour toute vari\'et\'e projective et lisse $X/\C$ avec $\Br(X)\neq 0$,
d'apr\`es  \cite{CT19b},
il existe des corps de fonctions d'une variable $F/\C$ tels que l'on ait $H^3_{nr}(F(X)/F, \Q/\Z(2)) \neq 0$. 
Pour $X/\C$ solide rationnellement connexe
satisfaisant  $\Br(X)\neq 0$, et $F/\C$ corps de fonctions d'une variable convenable, 
 (E) est donc en d\'efaut, et donc, d'apr\`es le th\'eor\`eme \ref{equivgenerales}, les hypoth\`eses
  (A), (B), (C), (D) le sont aussi.
 \end{rmk}

Pour les hypersurfaces cubiques lisses dans $\P^4_{\C}$, un couple $(S,Z)$ comme dans
le th\'eor\`eme \ref{solidesRC}
 est bien connu.

   \begin{thm}\label{equivcubiques}
  Soit $X \subset \P^4_{\mathbb  C}$ une hypersurface cubique lisse.
  Soit $S/\C$ la surface de Fano des droites trac\'ees sur $X$
  et $Z \subset S\times_{\C}X$ la correspondance associ\'ee. Soit $F/{\mathbb C}$ un corps.
 Avec les notations ci-dessus, les \'enonc\'es suivants sont \'equivalents.
 
  (A) L'application $A_{0}(S_{F}) \to \Alb_{S}(F)$ est surjective.
  
   (B) L'application $A^2(X_{F})  \to \Alb_{S}(F)$ est un isomorphisme.

(C) L'application  
$ A^2(X_{F}) \to  A^2(X_{\overline F})^G$  est un isomorphisme.

(D) L'application 
$ CH^2(X_{F}) \to  CH^2(X_{\overline F})^G$  est un isomorphisme.

(E) On a  $H^3(F,\Q/\Z(2))= H^3_{nr}(F(X)/F,\Q/\Z(2) )$.
  \end{thm}

\begin{proof}

Soit donc $Z \subset S \times X$ la vari\'et\'e d'incidence des droites contenue dans $X$.
Soit $p : Z \to S$ et $q: Z \to X$.  On a une application
$ q_{*} \circ p^* : A_{0}(S_{F}) \to A^2(X_{F})$
et un isomorphisme
$ p_{*} \circ q^* : CH^2(X)_{alg} \to \Pic^0(S).$
D'o\`u une application  $ A_{0}(S) \to \Pic^0(S)$
qu'on peut composer
avec  l'isomorphisme
$\Pic^0_{S} \simeq \Alb_{S}$
(\cite[Chap. 5, \S 3]{H23}).
Voir aussi \cite[\S 2B]{CTP18}.
 Mingmin Shen  \cite[Thm. 1.7, Thm. 4.1]{MSh19} a \'etabli le r\'esultat
 remarquable suivant  : pour tout corps $F/\C$,
l'application $A_{0}(S_{F}) \to A^2(X_{F})$  est surjective.
Ceci montre que (B) implique (A).

Pour $X/\C$ une hypersurface cubique lisse dans $\P^4_{\C}$, on a $\Br(X)=0$, et l'on sait que pour  $F/\C$ un corps,
on  a  $H^3_{nr}({\overline F}(X),\Q/\Z(2))=0$   \cite[Cor. 6.2]{CTV12}, \cite[Thm. 8.1, Cor. 8.2]{CTV12}.
 Il r\'esulte  alors  de
\cite[Lemme 5.7, Thm. 5.8]{CT15} avec la correction
  \cite[thm. 2.1]{CT19a} que l'on a 
$$H^3_{nr}(F(X), \Q/\Z(2))/H^3(F,\Q/\Z(2))  \simeq  \Coker[CH^2(X) \to CH^2(X_{{\overline F}})^G].$$
Ainsi (E) implique (D).

Les autres implications ont \'et\'e \'etablies dans le th\'eor\`eme \ref{equivgenerales}.
\end{proof}

 \begin{cor}\label{applic}
 Soit $X \subset \P^4_{\mathbb  C}$ une hypersurface cubique lisse.
 Soit $F/{\mathbb C}$ un corps de fonctions d'une variable.
 Alors toutes les propri\'et\'es du th\'eor\`eme \ref{equivgenerales}
 valent pour  $X_{F}$.
 \end{cor}
 \begin{proof}  Pour $F$ et $X$ comme dans l'\'enonc\'e, on a \'etabli dans \cite[Thm. 1.2]{CTP18}
  que l'on a  $  H^3_{nr}(F(X)/F,\Q/\Z(2))=0$.  QED
\end{proof}

 \begin{rmk} 
  Il existe une surface projective et lisse $Y/\mathbb C$ et un corps de fonctions d'une variable $F/\mathbb C$ 
 pour lesquels l'application $A_{0}(Y_{F}) \to \Alb_{Y}(F)$ n'est pas surjective.
 C. Voisin vient de donner un tel exemple.   On part d'une
vari\'et\'e $X$ qui est un solide double quartique tr\`es g\'en\'eral avec 7 points singuliers
 ordinaires, qui comme montr\'e dans \cite{V15} n'a pas de cycle de codimension deux universel. 
   La   jacobienne interm\'ediaire $J$ de $X$ est une vari\'et\'e ab\'elienne principalement polaris\'ee
  de dimension 3, c'est  la jacobienne d'une courbe $\Gamma$ de genre 3.
  On utilise \cite[Cor. 0.9]{V24b}.   
  On prend $Y=S$ une surface comme dans le th\'eor\`eme   \ref{solidesRC} ci-dessus, 
avec $\Alb_{S}=J$.   On prend pour $F$ le corps $\C(\Gamma)$.  On  montre que le plongement de $\Gamma$ dans $J$
   donne un point de $\Alb_{S}(F)$ qui n'est pas dans l'image de $A_{0}(S_{F})$.
\end{rmk}

 \begin{rmk}
 Soit $X$   comme dans le th\'eor\`eme \ref{equivgenerales}. Si $X$ est stablement rationnelle, ou r\'etractilement  rationnelle, ou simplement universellement $CH_{0}$-triviale, i.e. $deg : CH_{0}(X_{F}) \to \Z$
 est un isomorphisme pour tout corps $F/\C$, alors  
 $$H^3(F,\Q/\Z(2) )= H^3_{nr}(F(X)/F,\Q/\Z(2))$$  et
 toutes les propri\'et\'es du th\'eor\`eme valent pour tout corps  $F/\C$.
 On ne conna\^{\i}t pas d'exemple d'hypersurface cubique lisse dans $\P^4_{\mathbb C}$
pour laquelle les  propri\'et\'es du th\'eor\`eme   \ref{equivcubiques} sont en d\'efaut. 
On sait 
qu'il existe des hypersurfaces cubiques lisses  dans $\P^4_{\C}$
qui sont universellement $CH_{0}$-triviales, par exemple la cubique de Fermat.
Pour de telles hypersurfaces cubiques lisses, la surface de Fano $S$ des droites 
a donc la propri\'et\'e remarquable que pour tout corps $F/{\mathbb C}$,
l'application $A_{0}(S_{F}) \to \Alb_{S}(F)$ est surjective :  la surface de Fano
$S$ poss\`ede un z\'ero-cycle universel param\'etr\'e par
$\Alb_{S}$.
 \end{rmk}
 
 \medskip
 
 {\it  Je remercie Bruno Kahn, Federico Scavia,  Zhiyu Tian, Claire Voisin et Olivier Wittenberg pour diverses remarques sur cet article.}


\begin{thebibliography}{12}
 
 

 
\bibitem[ACMV23]{ACMV23} J. Achter, S. Casalaina-Martin, Ch. Vial,  A functorial approach to regular homomorphisms, Algebraic Geometry 10 (1) (2023) 87--129.

\bibitem[Bl84]{Bl84} S. Bloch, Height pairings for algebraic cycles, Journal of Pure and Applied Algebra 34 (1984) 119--145.

\bibitem[CT15]{CT15} J.-L. Colliot-Th\'el\`ene, Descente galoisienne sur le second groupe de Chow : mise au point et applications, Documenta math. Extra volume Merkurjev (2015) 195--220.

\bibitem[CT19a]{CT19a} J.-L. Colliot-Th\'el\`ene, Troisi\`eme groupe de cohomologie non ramifi\'ee des hypersurfaces de Fano, Journal tunisien de math\'ematiques (2019) vol. 1 no. 1, 47--57.

\bibitem[CT19b]{CT19b} J.-L. Colliot-Th\'el\`ene, Cohomologie non ramifi\'ee dans le produit avec une courbe elliptique,   manuscripta mathematica 160 (2019) no. 3--4, 561--565.    

\bibitem[CTI81]{CTI81} J.-L. Colliot-Th\'el\`ene et F. Ischebeck,  L'\'equivalence rationnelle sur les points ferm\'es
des vari\'et\'es alg\'ebriques r\'eelles, CRAS Paris t. 292 (1981) S\'erie I   723--725.



\bibitem[CTP18]{CTP18} J.-L. Colliot-Th\'el\`ene et A. Pirutka, Troisi\`eme groupe de cohomologie non ramifi\'ee
d'un solide cubique sur un corps de fonctions d'une variable. 
\'Epijournal de G\'eom\'etrie Alg\'ebrique
Volume 2 (2018), Article Nr. 13.

 

\bibitem[CTV12]{CTV12} J.-L. Colliot-Th\'el\`ene et C. Voisin, Cohomologie non ramifi\'ee et conjecture de Hodge enti\`ere, Duke Math. J. 161 (2012)  735--801.


 
\bibitem[Gr62]{Gr62} A. Grothendieck, Technique de descente et th\'eor\`emes d'existence en g\'eom\'etrie alg\'ebrique. VI. Les sch\'emas de Picard : propri\'et\'es
g\'en\'erales, Exp. No. 236, S\'eminaire Bourbaki, t. 14, 1961/1962, Soc. Math. France, Paris, 1995, pp. 221--243.

\bibitem[H23]{H23} D. Huybrechts,   The geometry of cubic hypersurfaces, Cambridge studies in advanced mathematics vol. 206 (2023) Cambridge University Press.

  
   \bibitem[K18]{K18} B. Kahn, Motifs et adjoints, Rend. Sem. mat. univ. Padova 139 (2018) 77--128.

\bibitem[K21]{K21} B. Kahn, Albanese kernel and Griffiths groups (avec un appendice de
Y. Andr\'e), Tunis J. Math. 3 (2021) 589--656.
 
   

   \bibitem[Kl05]{Kl05} S. Kleiman, The Picard scheme, in {\it Fundamental Algebraic Geometry}, Grothendieck's FGA explained,
   Mathematical Surveys and Monographs, vol. 123,  American Mathematical Society, 2005.
   
   
  \bibitem[KT24]{KT24} J. Koll\'ar et Zhiyu Tian,
  Stable maps of curves and algebraic equivalence of 1-cycles,
  Duke Math. Journal, to appear,  https://arxiv.org/pdf/2302.07069.
 
 
   
\bibitem[Li69]{Li69}   S. Lichtenbaum, Duality theorems for curves over $p$-adic fields, Invent. math. {\bf 7} (1969) 120--136.

\bibitem[MSh19]{MSh19} Mingmin Shen, Rationality, universal generation and the integral Hodge conjecture,
Geom. Topology 23  (2019) no. 6  2861--2898.
 


 

 
 

\bibitem[R72]{R72} 
A.A. Ro\v{\i}tman,
 Rational equivalence of zero-dimensional cycles,
Mat. Sb. (N.S.) 89(131) (1972), 569?585, 671. 
https://www.mathnet.ru/sm3248
traduction anglaise  : Rational equivalence of zero-dimensional cycles,
Math. USSR-Sb. 18 (1974), 571--588.

 

\bibitem[R80]{R80}  A.A. Ro\v{\i}tman,  The torsion of the group of zero-cycles modulo rational equivalence, Ann. of Math. 111 (1980)  553--569.

\bibitem[S59]{S59} J-P. Serre, Groupes alg\'ebriques et corps de classes, Actualit\'es scientifiques et industrielles 1264,  
Publications de l'institut math\'ematique de Nancago, VII  , Hermann, Paris (1959).
 
\bibitem[Vial15]{Vial15} Ch. Vial, Chow--K\"unneth decomposition for $3$- and $4$-folds fibered by varieties with trivial Chow group of zero-cycles, J. Algebraic Geometry 24 (2015) 51--80.


\bibitem[V02]{V02} C. Voisin, Th\'eorie de Hodge et g\'eom\'etrie alg\'ebrique complexe, Cours sp\'ecialis\'es 10, SMF (2002).

\bibitem[V15]{V15} C. Voisin,  Unirational threefolds with no universal codimension 2 cycle, Invent. math.  201 (2015) 207--237.

\bibitem[V24a]{V24a}   C. Voisin, Cycle classes on abelian varieties and the geometry of
the Abel-Jacobi map.    
PAMQ, Volume 20, Number 5  (volume in honour of Enrico Arbarello), pp. 2469--2496 (2024).
 

 \bibitem[V24b]{V24b}  C. Voisin, Geometric representability of 1-cycles on rationally connected threefolds,
   in {\it Perspectives on four decades of Algebraic Geometry : in Memory of Alberto Collino}, Progress in Mathematics, volume 352 (2024).
 

\bibitem[W34]{W34} E. Witt,  Zerlegung reeller algebraischer Funktionen in Quadrate. Schiefk\"{o}rper \"{u}ber reellem Funktionenk\"{o}rper,
J. reine angew. Math. 171 (1934) 4--11.
  
  \bibitem[W08]{W08} O. Wittenberg, On Albanese torsors and the elementary obstruction,
  Mathematische Annalen 340 (2008), no. 4, 805--838.
  
  \bibitem[W12]{W12} O. Wittenberg, 
  Z\'ero-cycles sur les fibrations au-dessus d'une courbe de genre quelconque,  
Duke Mathematical Journal 161 (2012), no. 11, 2113--2166.
  

\end{thebibliography}
\end{document}